\documentclass[12pt]{amsart}%

\usepackage{amsmath, amssymb, amsthm, amsfonts, bbm, dsfont}
\usepackage{graphicx}
\usepackage{float}
\usepackage{wrapfig}
\restylefloat{figure}
\usepackage{mathrsfs}

\usepackage{hyperref}

\usepackage{a4wide}
\usepackage{todonotes}

\usepackage{bm}

\numberwithin{equation}{section}
\theoremstyle{plain}
\newtheorem{theorem}{Theorem}
 
 \newtheorem{proposition}{Proposition}
 \newtheorem{corollary}{Corollary}

 \theoremstyle{definition}

\newtheorem{remark}{Remark}

\usepackage{stackengine}
\usepackage{mathtools}
\usepackage{amssymb}
\usepackage{dsfont}

\usepackage{mathabx}

\newcommand{\CC}{\mathbb{C}}

\newcommand{\PP}{\mathbb{P}}

\newcommand{\ZZ}{\mathbb{Z}}

\newcommand{\calV}{\mathcal{V}}
\newcommand{\calE}{\mathcal{E}}
\newcommand{\calF}{\mathcal{F}}

\newcommand{\calO}{\mathcal{O}}

\newcommand{\calL}{\mathcal{L}}
\newcommand{\calD}{\mathcal{D}}

\newcommand{\calK}{\mathcal{K}}

\newcommand{\lm}{\lambda}

\newcommand{\Gr}{\textup{Gr}}

\newcommand\Hom{\textup{Hom}}
\newcommand\End{\textup{End}}

\newcommand\GL{\textup{GL}}
\newcommand\gl{\mathfrak{gl}}

\newcommand{\eps}{\epsilon}
\newcommand{\lam}{\lambda}

\newcommand{\quash}[1]{}

\newcommand\mydots{\makebox[1em][c]{,\hfil.\hfil.}}

\usepackage{tikz}
\usetikzlibrary{shapes,fit,intersections}
\usetikzlibrary{decorations.markings}
\usetikzlibrary{decorations.pathreplacing}
\usetikzlibrary{patterns}
\usetikzlibrary{matrix,arrows}
\usepgflibrary{decorations.pathmorphing}

\usepackage{tikz-cd}

\newcommand{\calZ}{\mathcal{Z}}

\newcommand{\calY}{\mathcal{Y}}

\newcommand{\Fl}{\mathrm{Fl}}

\newcommand{\ot}{\otimes}

\title{A categorification of representations of $\Uqgloo$
  }

\author{A. Oblomkov}

\address{
A.~Oblomkov\\
Department of Mathematics and Statistics\\
University of Massachusetts at Amherst\\
Lederle Graduate Research Tower\\
710 N. Pleasant Street\\
Amherst, MA 01003 USA
}
\email{oblomkov@math.umass.edu}

\author{L. Rozansky}
\address{
L.~Rozansky\\
Department of Mathematics\\
University of North Carolina at Chapel Hill\\
CB \# 3250, Phillips Hall\\
Chapel Hill, NC 27599 USA
}
\email{rozansky@math.unc.edu}

\def\Uqglv#1{ U_q(\gl_{#1}) }
\def\Uqgloo{\Uqglv{1|1}}
\def\Uqglt{\Uqglv{2}}
\def\Uqglmn{\Uqglv{m|n}}
\def\glv#1{\gl_{#1}}
\def\glmn{ \glv{m|n} }
\def\gln{ \glv{n} }
\def\glN{ \glv{N}}
\def\gloo{ \glv{1|1}}
\def\glzt{ \glv{0|2}}
\def\gltz{ \glv{2|0}}
\def\glt{ \glv{2}}
\def\GLv#1{ \GL_{#1}}
\def\GLN{ \GLv{N}}
\def\VN{V^{\otimes N}}
\def\xH{ H }

\def\VNlm{ (\VN)_\lambda }

\def\grV{ \mathcal{V}}
\def\bfgrV{ \bm{\grV}}
\def\grVoN{ \grV^{\otimes N}}
\def\bflam{ \bm{\lam}}
\def\bfmu{ \bm{\mu}}

\def\xl{ l }
\def\Sl{\mathrm{Sl}}
\def\Slbmu{ \Sl_{\bfmu}}
\def\rmS{ \mathrm{S}}

\def\grX{ X }
\def\IC{ \mathbb{C}}

\def\glN{ \mathfrak{gl}(N)}

\def\tcnk{ \tilde{\mathfrak{n}}_k }
\def\rmTs{\mathrm{T}^* }

\def\YNlm{ Y(\lambda) }
\def\Csq{ \IC^* }

\def\ICoo{ \IC^{1|1} }

\def\ICmn{ \IC^{m|n} }
\def\xf{ f }

\def\GrkN{ \Gr(k,N) }

\def\vz{ v_0 }
\def\vo{ v_1 }

\def\cDlm{\calD(\lambda)}
\def\cDlmpt{ \calD(\lambda+2)}
\def\cDlmmt{ \calD(\lambda-2)}

\def\cElm{ \calE(\lambda) }
\def\cElmpt{ \calE(\lpt) }
\def\cFlm{ \calF(\lambda) }
\def\cFlmpt{ \calF(\lambda+2) }

\def\lmp{ \lam' }

\def\calH{\mathcal{H}}

\def\DbCs{ D^b_{\Csq}}

\def\bfq{\mathbf{q}}

\def\tXp{ \tilde{X}'}

\def\mtr#1{ \begin{bmatrix} #1 \end{bmatrix} }

\def\Wv#1{ W(#1) }

\def\Wlm{ \Wv{\lam,\lam+2}}
\def\Wlmp{ \Wv{\lam,\lmp}}
\def\Wlmlm{ \Wv{\lam,\lam}}

\def\cOWlm{ \calO_{\Wlm}}

\def\lpt{\lam + 2}

\def\xtu{ \tau}
\def\tul{  \xtu_\lam }
\def\tulpt{  \xtu_{\lpt}}

\def\ICN{ \IC^N }
\def\ICk{ \IC^k }

\def\cODl{ \calO_{\Delta} }

\def\FM{FM}

\begin{document}
\maketitle

\begin{abstract}

We categorify the action of $\Uqgloo$ on the tensor product of its vector representations $(\ICoo)^{\otimes N}$. The generators $E$ and $F$ are represented by Fourier-Mukai functors between the derived categories of coherent sheaves on the total spaces of "semi-parabolic" vector bundles over the Grassmannians $\Gr(k,N)$.

\end{abstract}

\section{Introduction}
\label{sec:introduction-}

Previously, Sartori~\cite{Sartori15} and Tian~\cite{Tian14} also categorified the $\Uqgloo$ action on
\((\CC^{1|1})^{\otimes N}\) by  different means: Sartori used the category $\calO$ and Tian used the symplectic geometry in the spirit of the Heegard homology of Ozsvath and Szabo.
Even before that Khovanov-Sussan \cite{KhovanovSussan16} used the diagrammatic approach to categorify the positive half of \(\Uqgloo\).
For most recent developments in categorical representation theory of \(\Uqgloo\) see \cite{ManionRouquier20}.
Our construction is of the algebro-geometric nature. For \(\Uqglt\) the algebro-geometric construction of the representation \((\CC^2)^{\otimes N}\) is presented by Cautis, Kamnitzer and Licata in \cite{CautisKamnitzerLicata10}.

\subsection{The quantum supergroup $\Uqgloo$ } 
We use the conventions and notations from the paper \cite{Sartori15}, for earlier discussion of this quantum
group see Appendix 1 of \cite{RozanskySaleur93}.
The weight lattice \(P\) of $\Uqgloo$ is spanned by \(\eps_1,\eps_2\) and the coweight lattice \(P^*\) is spanned by \(h_1,h_2\) with the natural pairing
\(\langle \eps_i,h_j\rangle=\delta_{ij}\). There are two simple roots \(\pm\alpha\), \(\alpha=\eps_1-\eps_2\) 
The algebra \(\Uqgloo\) has four generators \(E,F\) and \(q^{h_1},q^{h_2}\), but for convenience we trade the latter pair for
\[
K = q^{h_1+ h_2},\qquad \xH = q^{h_1 - h_2}. 
\]
The generators satisfy the relations
\begin{equation}\label{eq:FE0} 
  E^2=0,\quad  F^2=0,
\end{equation}
\begin{equation}
  \label{eq:EF} 
  EF+FE=K-K^{-1}, 
\end{equation}
\begin{equation}
  \label{eq:grad}
  \xH E = q^2 E \xH,\qquad \xH F = q^{-2} F\xH,
\end{equation}
\begin{equation}
  \label{eq:qh}
KE = EK,\quad KF = FK,\qquad K\xH = \xH K,
\end{equation}
so $K$ is central.

In order to obtain the standard relations for the quantum group \(\Uqgloo\), one has to rescale the generator $E$, that is, replace it by
\[\tilde{E}=E/(q-q^{-1}).
\]

The Hopf algebra structure of $\Uqgloo$ is described by the relations
\[\Delta(E)=E\ot K^{-1}+1\ot E,\quad \Delta(F)=F\ot 1+K\ot F,\]
\[S(E)=-EK,\quad S(F)=-K^{-1}\ot F,\]
\[\Delta(q^h)=q^h\ot q^h,\quad S(q^h)=q^{-h},\]
\[u(E)=u(F)=0,\quad u(q^h)=1,\]
where \(S\) and \(u\) are antipod and counit.



\subsection{Representation theory}
\label{sec:repr-theory}
Denote by $V=\IC^{1|1}$ the defining 2-dimensional representation of $\Uqgloo$. The action of the generators $E,F,K,H$ is described by the following matrices relative to the standard basis $\vz$ (even) and $\vo$ (odd):
\[
E = \mtr{0 & q - q^{-1}\\ 0 & 0},\quad 
F = \mtr{ 0 & 0 \\ 1 & 0},\quad 
K = \mtr{q & 0 \\ 0 & q},\quad
H = \mtr{q & 0 \\ 0 &  q^{-1}}.
\]


The action of the generators on the tensor power $\VN$ is determined by the comultiplication $\Delta$. 


The generator $K$ acts on $\VN$ by the multiplication by $q^N$, and $\VN$ splits into the weight spaces according to the eigenvalues of $\xH$:
%
\begin{equation}
\label{eq:wght}
\VN = \bigotimes_{k=0}^N (\VN)_{N-2k}, 
\end{equation}
the space $\VNlm$ (for $\lambda=N,N-2,\ldots,-N$) being the eigenspace of $\xH$ with the eigenvalue $q^{\lambda}$.
We use notation \(E(\lm)\), \(F(\lm)\) for the corresponding linear maps
\[E(\lm)\colon (V^{\otimes N})_\lm\to (V^{\otimes N})_{\lm+2},\qquad   F(\lm)\colon (V^{\otimes N})_\lm\to (V^{\otimes N})_{\lm-2},\]
\[ H(\lm),K(\lm)\colon (V^{\otimes N})_\lm\to (V^{\otimes N})_{\lm}\]

%
%


 \subsection{Categories}

Throughout the paper we maintain the relation
\[
\boxed{\lambda = N - 2k}
\]
where $\lambda = -N,\ldots,N$ is the weight index of~\eqref{eq:wght}, while $k =  0,\ldots,N$ is the parameter of the Grassmannian $\GrkN$
in the constructions below.

Our categorification of $\gloo$ is based on Calabi-Yau varieties $\YNlm$ which, similar to $\rmTs\GrkN$, 
are the total spaces of  vector bundles $\YNlm\rightarrow\GrkN$: 
\[
\YNlm=\{(\grV,\grX)\,|\,\grV\subset \IC^N,\;\dim \grV = k,\;\grX\in\Hom(\IC^N,\grV)\},
\]
that is, the fibers of $\YNlm$ are half-nilpotent and half-parabolic.
Equivalently,
\[
\YNlm= \bigl(\GLN\times\tcnk\bigr)/P_k,
\]
where 
\[
\tcnk =\{X\in\glN\;|\;X(\ICN)\subset\ICk\},
\]
$\ICk\subset\ICN$ being the standard $k$-dimensional subspace, and $P_k\subset\GLN$ is the parabolic subgroup:
\[
 P_k = \{
g\in \GLN\;|\;g(\ICk)\subset \ICk\}.
\]

%

The group $\Csq$ acts on $\YNlm$ by scaling the fibers: $\grX\mapsto t^2 \grX$.
To the weight space $\VNlm$ we associate 
 the 
derived category of
\(\Csq\)-equivariant bounded complexes of coherent sheaves over $\YNlm$:
\begin{equation}
\label{eq:cats}
\calD(\lambda)=D_{\CC^*}^b\bigl(Y(\lambda)\bigr).
\end{equation}

%

%
%

\subsection{Functors}
\label{sec:functors}

Having selected the categories~\eqref{eq:cats}, we choose the functors
\[
\begin{tikzcd}
\cDlmmt && \cDlm \arrow[rr,"\calE(\lm)"] \arrow[ll,"\calF(\lm)"'] 
\ar[looseness=6, out=300, in=240, "{\calK(\lm),\calH(\lm)}"]
& &
\cDlmpt
\end{tikzcd}
\]
corresponding to (the matrix elements of) the generators $E,F,K,H$ of $\Uqgloo$ in such a way that these functors satisfy relations categorifying
(up to a shift) the relations~
\eqref{eq:FE0}-\eqref{eq:qh}.

Let $\bfq$ denote the endo-functor of $\cDlm$ that shifts the weights of the $\Csq$ action by  \(1\). Then the obvious choice for $\calH(\lm)$ and $\calK(\lm)$ is
\(
\calK(\lm) = \bfq^\lambda,\) \( \calH(\lm) = \bfq^N.
\)

In the paper we use the standard Fouries-Mukai functor formalism. In particular, an object \(\calZ\in D^b_{\CC^*}\bigl(\calY(\lm),\calY(\lm')\bigr)\) yields
a functor \(\calD(\lm)\to \calD(\lm')\). We denote the functor also as \(\calZ\) in order to keep the notations simple. The sheaf \(\calZ\) is
called the FM kernel of the corresponding functor.

Now we define the FM kernels for the functors \(\calE(\lm)\) and \(\calF(\lm)\).
For two weights $\lam,\lmp$ such that $\lmp\geq\lam$ and $\lmp-\lam$ is even, define a smooth subvariety $\Wlmp\subset Y(\lam)\times Y(\lmp)$:
\[
\Wlmp=\{ (\grV,X,\grV',X')\;|\;\grV\subset \grV', X'=X\}, \]
where the condition $X'=X$ has the following precise meaning. Since $\grV$ is required to be a subspace of $\grV'$, we can further require $p\circ X' =0$, where $p\colon \grV'\rightarrow \grV'/\grV$. Thus $X'\colon \IC^N\rightarrow \grV'$ is reduced to $\tXp\colon \IC^N\rightarrow \grV$, and we finally require $\tXp=X$.


Let $\tul\rightarrow Y(\lam)$ denote the pull-back of the tautological bundle over $\GrkN$: the fiber of $\tul$ over a point $(\grV,\grX)\in Y(\lam)$ is the vector space $\grV$. We define
\begin{gather*}
\cElm  = \cOWlm \otimes \pi^*_{\lambda} (\det \tau_\lambda)
\\
\calF(\lambda+2)  =  \cOWlm \otimes \pi^*_{\lam} (\det \tul)^{N-k} \otimes \pi^*_{\lpt} (\det\tulpt)^{-N+k}\,
\end{gather*}


Finally, note that the (pushed forward) structure sheaf $\cODl$ of the diagonal
$\Delta=\Wlmlm\subset Y(\lam)\times Y(\lam)$ is the FM kernel of the identity functor.


The following is our main result:
\begin{theorem}
  Thus defined, the sheaves $\cElm$ and $\cFlm$ satisfy the following relations: 
\begin{equation}\label{eq:nilp}
\cElmpt \circ \cElm \simeq 0,\qquad \cFlm \circ \cFlmpt \simeq 0,
\end{equation}
\begin{equation}\label{eq:cone}
  [\calO_{\Delta}[N-k-1]  \xrightarrow{\;\;\alpha\;\;}     \calF(\lam+2) \circ \calE(\lam)]
\simeq
  [\mathbf{q}^{2N}\calO_{\Delta}[k-N-1] \xrightarrow{\;\;\beta\;\;}
  \calE(\lam-2)\circ \calF(\lam)] 
  \end{equation}
where $\alpha$ and $\beta$ are the appropriate morphisms between the corrresponding sheaves to be specified in the proof.
\end{theorem}

The decategification of the above action is obtained by passing to the localized
\(\CC^*\)-equivariant \(K\)-theory of categories \(\calD(\lm)\):

\begin{corollary} The  \(2^N\)-dimensional  vector space \[ \bigoplus_{\lambda} \mathrm{K}^{\mathrm{loc}}_{\CC^*}\bigl(Y(\lm)\bigr) \]
  over the field of fractions of \(\mathrm{K}_{\CC^*}(\mathrm{pt})\) is a representation of \(\Uqgloo\) via:
  \[E(\lambda)\mapsto \mathbf{q}^{-N}\calE(\lambda),\quad F(\lm)\mapsto \calF(\lm)[N-k],\quad K(\lm)\mapsto \mathbf{q}^\lm \calO_\Delta,
    \quad H(\lm)\mapsto \mathbf{q}^N\calO_{\Delta}.
  \]  
\end{corollary}

By comparing the dimensions of the weight spaces one can conclude that the \(\Uqgloo\)-representation from the corollary above
is exactly \((\CC^{1|1})^{\otimes N}\).

\begin{remark}
  We use the cohomological conventions throughout the paper. The term $C_i$ of a chain complex
  $\cdots C_2\to C_1 \to C_0$ has the cohomological degree $-i$ and in our shift conventions 
   \((C[n])_k=C_{k-n}\).
  
\end{remark}

This note has three sections. We present a proof of the main theorem in section~\ref{sec:mainthm}. In section~\ref{sec:next} we present a proposal
that extend our construction to the case of \(\Uqglmn\).

{\bf Acknowledgments}
We would like to thank Jenia Tevelev
for useful discussions.
The work of A.O. was supported in part by  the NSF grant DMS--2200798.
The work of L.R. was supported in part by  the Simons Collaboration Grant \emph{"New Structures in Low-Dimensional Topology"} 5126097.

\section{Proof of the main theorem}
\label{sec:mainthm}


\subsection{Nilpotent relation}
\label{sec:nilpotent-relation}

Let us first prove the relations \eqref{eq:nilp}. The kernels of the compositions \(\calE(\lambda)\circ \calE(\lambda-2)\)
and of \(\calF(\lambda)\circ\calF(\lambda+2)\) are the push-forwards along the projection \(\pi_{\lambda-2,\lambda+2}\) from
\(Y(\lambda-2)\times Y(\lambda)\times Y(\lambda+2) \) of the kernels
\[\calO_{W(\lambda-2,\lambda,\lambda+2)}\otimes \pi^*_\lambda(\det\tau_\lambda)\otimes\pi_{\lambda+2}^*(\det\tau_{\lambda+2}) ,\]
\[
  \calO_{W(\lambda-2,\lambda,\lambda+2)}\otimes\pi^*_{\lambda-2}(\det\tau_{\lambda})^{N-k-1}\otimes \pi^*_\lambda(\det\tau_\lambda)\otimes\pi_{\lambda+2}(\det\tau_{\lambda+2})^{k-N},\]
where the variety \(W(\lambda-2,\lambda,\lambda+2)\) consists of the triples \((\calV,X),(\calV',X'),(\calV'',X'')\)
inside \(Y(\lambda-2)\times Y(\lambda)\times Y(\lambda+2)\)
such that
\[\calV\subset \calV'\subset \calV'', \quad X=X'=X''.\]

Over a point  of \( Y(\lambda-2)\times Y(\lambda+2)\) the fiber of the projection \(\pi_{\lambda-2,\lambda+2}\) intersected with \(W(\lambda-2,\lambda,\lambda+2)\) is a \(\PP^1\). Moreover, the restriction of the above described FM kernel on this \(\PP^{1}\) is
a line bundle \(\calO(-1)\) that has no homology. Hence the push-forward along \(\pi_{\lambda-2,\lambda+2}\) of both kernels are homotopic to zero.

\subsection{Koszul complexes and generalizations}
\label{sec:kosz-compl-gener}

We use  generalization of the Koszul complexes in our arguments. Here we fix notations for these complexes and corresponding differentials.
For a given vector bundle \(V\) on variety  \(X\) a section \(s\in H^0(X,\calV)\)
defines a the zero-section subvariety \(Z(s)\subset X\).
Let \(V^\vee\) be a dual vector bundle then the  complex 
Respectively, the Koszul complex \(K(V^\vee,s)=\oplus_{j=0}^rK_{-j}(V^\vee,s)\), \(K_{-j}(V^\vee,s)=\Lambda^jV^\vee\)
is a complex of vector bundles such that the zeroth  homology of this complex is the sheaf \(\calO_{Z(s)}\). Moreover, if the
intersection is transverse then \(\calO_{Z(s)}\) is homotopy equivalent to the complex \(K(V^\vee,s)\).

To define the differential in \(K(V^\vee,s)\) a cosection \(\check{s}: \calV^\vee\to \calO_X\) that is dual to \(s\).
On a sufficiently small affine chart \(U\subset X\) we can trivialize \(V^\vee\) by choosing
sections \(v_i\in H^0(U,V^\vee)\), \(i=1,\dots, \mathrm{rank}(V)\) that span every fiber of \(V^\vee\) over \(U\).
Then the differential \(D_{s}^j:\Lambda^jV^\vee\to \Lambda^{j-1}V^\vee \) is defined as:
\[ D^j_{s}v_{i_1}\wedge\dots \wedge v_{i_s}= \sum_{l=1}^s (-1)^l\check{s}(v_{i_l})v_{i_1}\wedge\dots \widehat{v}_{i_l}\dots \wedge v_{i_s}.
   \]

   In our arguments below we need a slight generalization of the Koszul complex above. Let us fix a line bundle \(\calL\), \(t\in H^0(X,\calL)\). The generalizations that we need interpolate between the Koszul complexes \(K(V^\vee,s)\) and \(K(V^\vee\otimes \calL^\vee,s\cdot t)\). For \(I\subset \ZZ\) we define \(K^I(\calL,t;V^{\vee},s)\) as complex with
  \[K^I_{-j}(\calL,t;V^\vee,s)=\Lambda^jV^{\vee}\otimes (\calL^\vee)^{d(I,j)},\quad d(I,j)=|I\cap [1,j]|. \]
  and the differential \[D_{t;s}^{I;j}: K^I_{-j}(\calL,t;V^\vee,s)\to K^I_{-j+1}(\calL,t;V^\vee,s)\]
  is defined by
  \[D_{t;s}^{I;j}=\begin{cases} D_s^j,& j\notin I\\
    tD^j_{s},& j\in I
  \end{cases}.\]

If the sections of the vector bundles are clear from the context then we abbreviate the notation for the generalized
Koszul complex to \(K^I(\calL;\calV^\vee)\).

\subsection{Highest weight space}
\label{sec:exter-weight-spac}

First, let us compute \(\calE(N-2)\circ\calF(N)\). The space \(Y(N)\) is a point and \(Y(N-2)=\calO(-1)^{\oplus N}\) where
\(\calO(-1)\) is a tautological line bundle over \(\PP^{N-1}\).
Respectively, the correspondence variety  \(W(N-2,N)\subset Y(N)\times Y(N-2)\) is the zero section \(\PP^{N-1}\subset Y(N-2)\).
Thus the FM kernel for \(\calE(N-2)\circ \calF(N)\) is the sheaf:
\[\calO_{\PP^{N-1}}\otimes \calO_{\PP^{N-1}}\otimes \det(\tau_{N-2})^{-N}.\]

Let \(\pi: Y(N-2)\to \PP^{N-1}\) is the projection on \(\PP^{N-1}\).
There is a natural cosection \(\check{s}_\tau: \pi^*\calO(1)^{\oplus N}\to \calO\), that is dual to the section \(s\), and
\(\calO_{\PP^{N-1}}=K(\pi^*\calO(1)^{\oplus N},s_\tau)\). Thus \[\calO_{\PP^{N-1}}\otimes \calO_{\PP^{N-1}}=\oplus_{j=0}^N \mathbf{q}^{2j}\calO(j)\otimes \calO_{\PP^{N-1}}\otimes\Lambda^j\CC^N.\]
Since \(\det(\tau_{N-2})^{-N}=\calO(-N)\), the convolution \(\calE(N-2)\circ \calF(N)\) is the direct sum
\[\calE(N-2)\circ \calF(N)\simeq \calO[N-1]\oplus \mathbf{q}^{2N}\calO[-N].\]

\subsection{Relation between the cones, \(\lambda\ne \pm N\)} 
First, let us study the functor \(\calF(\lm+2)\circ\calE(\lm)\). The kernel of this functor is the push-forward along
the map \(\pi_{\lm',\lm''}\) of the sheaf \(\calO_{FE}\) on \(Y(\lm')\times Y(\lm+2)\times Y(\lm'')\):
\[\calO_{FE}=\calO_{W(\lm',\lm+2,\lm'')}\otimes\pi^*_{\lm'}(\det\tau_{\lm'})\otimes \pi_{\lm''}^*(\tau_{\lm''})^{N-k}\otimes\pi_{\lm+2}^*(\tau_{\lm+2})^{k-N},\]
here and everywhere below we use notation \(\lm'\) and \(\lm''\) for the first and second copy of \(\lm\) to distinguish them.
In the last formula the scheme \(W(\lm',\lm+2,\lm'')\) is defined as a set of triples \((\calV'_{k-1},X'),(\calV_{k},X),(\calV''_{k-1},X'')\)
that satisfy:
\begin{equation}\label{eq:W-FE}
  \calV_{k-1}\subset \calV'_k,\quad \calV_{k-1}\subset\calV''_{k},\quad X=X'=X''.
\end{equation}
We use subindices indicate the dimensions of corresponding spaces.
By the dimension count one can see that \(W(\lm',\lm+2,\lm'')\) is a transverse intersection of \(W(\lm',\lm+2)\times Y(\lm'')\)
and \(Y(\lm')\times W(\lm+2,\lm'')\). In particular, the scheme \(W(\lm',\lm+2,\lm'')\) is vector bundle over the smooth space
inside \(\Gr(k,N)\times\Gr(k-1,N)\times\Gr(k,N)\) defined as 
\[Z^+=\{(\calV'_k,\calV_{k-1},\calV''_{k})|\calV_{k-1}\subset \calV'_k,\calV_{k-1}\subset \calV''_k\}.\]

Let us denote by \(M^+\) the total space of the sum of vector bundles \((\CC^N)^\vee\otimes \tau_{\lm'}\),
  over the space \(Z^+\).
We use notation \(\tau_{\lm'}\), \(\tau_{\lm+2}\) and \(\tau_{\lm''}\) for the pull-backs of the corresponding line bundles on
\(Z^+\) to \(M^+\).  Let us also define the map \(p_+:M^+\to Y(\lm')\times Y(\lm'')\) by the formula:
\[p_+(\calV'_k,\calV_{k-1},\calV''_k;X)=((\calV'_k,X),(\calV'_k,X)).\]

Over \(M^+\) the vector bundle \((\CC^N)^\vee\otimes \tau_{\lm'}/\tau_{\lm+2}\) 
have tautological section \(s_{\lm',\lm+2}\) and \(s_{\lm'',\lm+2}\) of these vector bundles. The space \(\hat{W}(\lm',\lm+2,\lm'')\) is
a zero locus of this section and it is isomorphic the space \(W(\lm',\lm+2,\lm'')\).
Let us denote by \(\widehat{\calO}_{FE}\in D(M^+)\) the corresponding sheaf:
\begin{equation}\label{eq:kerFE}
  \widehat{\calO}_{FE}=\calO_{\hat{W}(\lm',\lm+2,\lm'')}\otimes \pi^*(\det \tau_{\lm'})^{k+1-N}\otimes\pi^*(\det \tau_{\lm-2})^{N-k},
\end{equation}

Thus the projection map \(p_+: M^+\to Y(\lm')\times Y(\lm'') \) is proper and \(\calF(\lm+2)\circ\calE(\lm)\) has FM kernel
\(p_{+*}(\widehat{\calO}_{FE})\):
\begin{equation}
  \label{eq:kerFE-2}
  \calF(\lm+2)\circ\calE(\lm)=
p_{+*}(\widehat{\calO}_{FE})
\end{equation}

Second, we study the functor \(\calE(\lm-2)\circ\calF(\lm)\). The kernel of this functor is the push-forward along
the map \(\pi_{\lm',\lm''}\) of the sheaf \(\calO_{EF}\) on \(Y(\lm')\times Y(\lm-2)\times Y(\lm'')\):
\[\calO_{EF}=\calO_{W(\lm',\lm-2,\lm'')}\otimes  \pi^*(\det \tau_{\lm'})^{k+1-N}\otimes\pi^*(\det \tau_{\lm-2})^{N-k}.\]
Here the scheme \(W(\lm',\lm-2,\lm'')\) is defined as a set of triples  \((\calV'_{k},X'),(\calV_{k-1},X),(\calV''_{k},X'')\)
that satisfy:
\begin{equation}\label{eq:W-EF}
  \calV_{k+1}\supset \calV'_k,\quad \calV_{k+1}\supset\calV''_{k},\quad X=X'=X''.
\end{equation}
We use subindices indicate the dimensions of corresponding spaces.
By the dimension count one can see that \(W(\lm',\lm-2,\lm'')\) is not a transverse intersection of \(W(\lm',\lm-2)\times Y(\lm'')\)
and \(Y(\lm')\times W(\lm-2,\lm'')\).
The scheme \(W(\lm',\lm-2,\lm'')\) is vector bundle over the smooth space
inside \(\Gr(k,N)\times\Gr(k+1,N)\times\Gr(k,N)\) defined as 
\[Z^{-} = \{(\calV'_k,\calV_{k+1},\calV''_{k})|\calV_{k+1}\supset \calV'_k,\calV_{k+1}\supset \calV''_k\}.\]

Let us denote by \(M^-\) the total space of the sum of the vector bundles \((\CC^N)^\vee\otimes \tau_{\lm'}\),
\((\CC^N)^\vee\otimes \tau_{\lm-2}\)  over the space \(Z^-\). We use notation \(\tau_{\lm'}\), \(\tau_{\lm-2}\) and \(\tau_{\lm''}\) for the pull-backs of the corresponding line bundles on
\(Z^-\) to \(M^-\).

Let us denote by
\(\hat{W}(\lm',\lm-2,\lm'')\) the isomorphic image of \(W(\lm',\lm-2,\lm'')\) inside
\(M^-\). Respectively, there is an object  \(\widehat{\calO}_{EF}\in D(M^-)\):
\begin{equation}\label{eq:kerEF}
  \widehat{\calO}_{EF}=\calO_{\hat{W}(\lm',\lm-2,\lm'')}\otimes \pi^*(\det \tau_{\lm'})^{k+1-N}\otimes\pi^*(\det \tau_{\lm-2})^{N-k}.
\end{equation}

We define the map \(p_-: M^-\to Y(\lm')\times Y(\lm'')\) by the formula:
\[p_-(\calV'_k,\calV_{k+1},\calV''_k;X)=((\calV'_k,X),(\calV''_k,X)).\]
The projection map \(p_-\) proper and \(\calE(\lm-2)\circ\calF(\lm)\) has FM kernel
\(p_{-*}(\calO_{EF})\):
\begin{equation}
  \label{eq:kerEF-2}
 \calE(\lm-2)\circ\calF(\lm)=
p_{-*}(\calO_{EF}).
\end{equation}

The spaces \(Z^+\) and \(Z^-\) have a common blow-up which we denote by \(Z^{\pm}\). This space is defined as follows:
\[Z^{\pm}=\{V_{k-1}\subset \calV'_k,\calV''_k\subset \calV_{k+1}\}.\]
Respectively, we define \(M^{\pm}\) to be the total space of  the vector bundle
\((\CC^N)^\vee\otimes \tau_{\lm'}\) over
\(Z^{\pm}\).

Let us denote the maps to from \(M^{\pm}\) to \(M^+\) and \(M^-\) by \(\pi_+\) and \(\pi_-\), respectively.
These maps are defined as
\[\pi_{+}(\calV_{k-1},\calV'_k,\calV''_k,\calV_{k+1};X)=(\calV'_k,\calV''_k,\calV_{k- 1},X),\]
\[\pi_{-}(\calV_{k-1},\calV'_k,\calV''_k,\calV_{k+1};X)=(\calV'_k,\calV''_k,\calV_{k+1},X).\]

Our strategy relies on comparison of the pull-backs of \(\widehat{\calO}_{FE}\) and \(\widehat{\calO}_{EF}\) on \(M^{\pm}\).
Thus, we
denote by \(W(\lm',\lm-2,\lm'')^{\pm}\) and \(W(\lm',\lm+2,\lm'')^{\pm}\) the preimages of the corresponding varieties inside
\(M^{\pm}\). Both of these varieties are zero loci of some natural vector bundles that we detail below.

The space \(M^\pm\) is a total space of the vector bundle \((\CC^N)^\vee\otimes \tau_{\lm-2}/\tau_{\lm'}\) over \(Z^\pm\)
Let \(\pi:M^\pm\to Z^\pm\) be the projection. Then \(\pi^*((\CC^N)^\vee\otimes \tau_{\lm-2}/\tau_{\lm'})\) has a tautological
section \(s_{\tau}\) that take value \(X\) at a point \((\calV_{k-1},\calV'_k,\calV''_k,\calV_{k+1};X)\).
The variety \(W(\lm',\lm+2,\lm'')^{\pm}\) is a zero locus of the the tautological section of the vector bundle
\((\CC^N)^\vee\otimes \tau_\lm/\tau_{\lm+2}\).

Inside, \(M^\pm\) we have an exceptional divisor \(E=\PP^{k-1}(\tau)^\vee\times \PP^{N-k-1}(\CC^N/\tau)\)
where \(\tau\) is the tautological bundle over \(\Gr(k,N)\). Other words,
the exceptional divisor fibers over \(\Gr(k,N)\) with fibers \(\PP^{k-1}\times \PP^{N-k-1}\).
The map \(\pi_+\) sends \(E\) to \(\PP^{k-1}(\tau)^\vee\subset M^+\) which is the locus defined by the condition \(\calV'=\calV''\).
Similarly, the map \(\pi_-\) sends \(E\) to \(\PP^{N-k-1}(\CC^N/\tau)\subset M^-\) which is the locus defined by the same  condition \(\calV'=\calV''\).

 The line bundle  \(\Hom(\pi^*(\tau_{\lm'}/\tau_{\lm+2}),\pi^*(\tau_{\lm-2}/\tau_{\lm''}))\)
 has a tautological section \(s_\Delta\) whose zero-locus is the exceptional divisor. Thus we have a isomorphism of line bunldes
  \begin{equation}\label{eq:O(delta)}
   \calO(E)\simeq \Hom(\pi^*(\tau_{\lm'}/\tau_{\lm+2}),\pi^*(\tau_{\lm-2}/\tau_{\lm''}))\simeq
   \det\tau_{\lm+2}\otimes\det\tau_{\lm-2}\otimes
   \det\tau^\vee_{\lm'}\otimes \det\tau^\vee_{\lm''}.
 \end{equation}
 In the complement of \(E\) the variety \(W(\lm',\lm+2,\lm'')^{\pm}\) is isomorphic to \(W(\lm',\lm-2,\lm'')^{\pm}\) and that is reflected by the fact that \(W(\lm',\lm-2,\lm'')^\pm\) is the zero locus of
 the section \(s_\tau\cdot s_\Delta\).

  Thus we have shown that we have closely related presentations of the corresponding FM kernels
 \(\pi^{+*}(\widehat{\calO}_{FE})\) and  \(\pi^{-*}(\widehat{\calO}_{EF})\)
 \[ \Lambda^\bullet(\pi^*(\CC^N\otimes (\tau_{\lm'}/\tau_{\lm+2})^\vee),\check{s}_\tau)\otimes \pi^*(\det\tau_{\lm'})
\otimes \pi^*(\det \tau_{\lm''})^{N-k}\otimes\pi^*(\det \tau_{\lm+2})^{k-N}
 \]
 \[\Lambda^\bullet(\pi^*(\CC^N\otimes (\tau_{\lm'}/\tau_{\lm+2})^\vee(-E)),\check{s}_\tau\cdot \check{s}_\Delta)
 \otimes \pi^*(\det \tau_{\lm'})^{k+1-N}\otimes\pi^*(\det \tau_{\lm-2})^{N-k}.
 \]

 Now  we need to relate the Koszul complexes with differentials
 \(s_\tau\) and \(s_\tau\cdot s_\Delta\). For that we rely on the following homological algebra statement.

 \begin{proposition}
   Let \(V\) and \(\calL\) be a vector bundle and a line bundle on \(X\),
   \(g\in H^0(X,V),f\in H^0(X,\calL)\).
   Suppose that \(i\in I\) and \(i+1\notin I\) then there is
   a complex morphism
   \[\varphi_{i,I}:  (K(\calL^\vee,f)\otimes(\calL^\vee)^{d(i)} \otimes\Lambda^iV^\vee)[1-i] \to K^I(\calL,f;V^\vee,g),\]
   where \(d(i)=d(i,I)\),
   such that for the cone \(C(\phi_{i,I})\) we have the homotopy equivalence:
   \begin{equation}\label{eq:cone-add}
     C(\varphi_{i,I})\simeq K^{I'}(\calL,f;V^\vee,g), \quad I'=(I\setminus \{i\})\cup \{i+1\}.
   \end{equation}


 \end{proposition}
 \begin{proof}
   The map in the statement is given by the diagram below:
\[\begin{tikzcd}[cramped,sep=tiny]
	\cdots & {\Lambda^{i+1}V^\vee\otimes (\mathcal{L}^\vee)^{d(i)}} & {} & {\Lambda^{i}V^\vee\otimes (\mathcal{L}^\vee)^{d(i)}} && {\Lambda^{i-1}V^\vee\otimes (\mathcal{L}^\vee)^{d(i)-1}} && \cdots \\
	\\
	&&& {\Lambda^{i}V^\vee\otimes (\mathcal{L}^\vee)^{d(i)}} && {\Lambda^{i}V^\vee\otimes (\mathcal{L}^\vee)^{d(i)-1}}
	\arrow["{D_g}"', from=1-2, to=1-4]
	\arrow["{fD_g}"', from=1-4, to=1-6]
	\arrow[from=1-6, to=1-8]
	\arrow["1"', from=3-4, to=1-4]
	\arrow["f", from=3-4, to=3-6]
	\arrow["{D_g}", from=3-6, to=1-6]
\end{tikzcd}\]
after we contract the isomorphism arrow.
\end{proof}



We need an iterated version of the above result for two presentations of the complex
\(K^{[1,N-k]}(\calL;V^\vee)\), \(\mathrm{rank}(V)=N\):
\begin{corollary}
  There are two iterated cone presentations for the complexes \(K^{[1,N-k]}(\calL,f;V^\vee,g)\)
  and \(K^{[1,N-k]}(\calL,f;V^\vee,g)\otimes (\calL^\vee)^{k-N}\), respectively:
  \begin{equation}\label{eq:cone-}
  C(K(V\otimes \calL,f g);\psi^-_{1,N},\psi_{1,N-1}^-\mydots,\psi^-_{1,N-k},\psi^-_{2,N}\mydots,\psi^-_{2,N-k+1}\mydots,
    \psi^-_{k-1,N},\psi^-_{k-1,N-1},\psi^-_{k,N}),
  \end{equation}
  \begin{equation}\label{eq:cone+}
  C(K(V,g);\psi^+_{1,1},\psi^+_{1,2}\mydots,\psi^+_{1,N-k},\psi^+_{2,1}\mydots,\psi^+_{2,N-k-1}\mydots,\psi^+_{N-k-1,1},
    \psi^+_{N-k-1,2},\psi^+_{N-k,1}),
  \end{equation}
  where  the iterated cones are defined inductively:
  \[C(K;)=K,\quad C(K;\varphi_1,\dots,\varphi_{m+1})=[C(K;\varphi_1,\dots,\varphi_m)\xleftarrow{\varphi_{m+1}}K(\calL^\vee,f)\otimes W_{\varphi_{m+1}}],\]
  for the vector bundles \(W_{\varphi}\) that are defined by:
  \[W_{\psi^+_{i,j}}=\mathbf{q}^{2j}(\calL^\vee)^{-i}\otimes\Lambda^jV[-j],\quad W_{\psi^-_{i,j}}=\mathbf{q}^{2j}(\calL^\vee)^{j-i}\otimes \Lambda^jV[-j] \]
\end{corollary}

Now we apply the corollary for the case \(X=Z^{\pm}\) and \(V=\pi^*(\CC^N\otimes (\tau_{\lm'}/\tau_{\lm+2})^\vee)\) and
\(g=s_\tau\), \(\calL=\calO(E)\), \(f=s_\Delta\). Thus we have two iterated cone presentations for the complex
\[\widehat{K}=K^{[1,N-k]}(\calL,f;V^\vee,g)\otimes (\calL^\vee)^{k-N}\otimes (\det\tau_{\lm+2})^{k-N}\otimes
  (\det\tau_{\lm''})^{N-k}\otimes \det\tau_{\lm'}.\] Below we apply push-forwards \(\pi^+_*\) and
\(\pi^-_*\) to these two presentations.

First let us compute the push-forward  \(\pi_{+*}(\widehat{K})\). The complex \(K(V,g)\otimes (\det \tau_{\lm+2})^{k-N}\otimes (\det\tau_{\lm''})^{N-k}\otimes \det\tau_{\lm'}\) is one of the
ends of the iterated cone. Since  the last complex is exactly \(\widehat{\calO}_{FE}\) and \(\pi_{+*}(\calO)=\calO\), we conclude that 
\(p_{+*}\circ \pi_{+*}(\widehat{K})\) contains \(\calF\circ\calE\) at one end of the push-forward iterated cone.

Similarly we can analyze \(\pi_{-*}(\widehat{K})\). Indeed, the end of the other iterated cone presentation is the complex \(K(V\otimes \calL,f\cdot g)\otimes (\det\tau_{\lm-2})^{N-k}\otimes (\det\tau_{\lm'})^{k-N+1}\).  The last complex is \(\widehat{\calO}_{EF}\) and since \(\pi_{-*}(\calO)=\calO\), we conclude that 
\(p_{-*}\circ \pi_{-*}(\widehat{K})\) contains \(\calE\circ\calF\) at one end of the push-forward iterated cone.

What is left for us to show that the push-forwards \(p_{+*}\circ\pi_{+*}\) and \(p_{-*}\circ \pi_{-*}\) annihilate  most of the rest of the terms of the
iterated cones. In details, the iterated cone \eqref{eq:cone-} twisted by  \((\det\tau_{\lm-2})^{N-k}\otimes (\det\tau_{\lm'})^{k-N+1}\) has the following
terms, other than the extreme end discussed above:
\begin{multline}\label{eq:fib-}
(\calL^\vee)^{j-i}\otimes(\det\tau_{\lm-2})^{N-k}\otimes (\det\tau_{\lm'})^{k-N+1}\otimes\calO_E\otimes\Lambda^jV[-j]
  \simeq (\det \tau_{\lm-2})^{N-k-j+i}\\ \otimes (\det \tau_{\lm+2})^{i-j}\otimes \calO_E\otimes \Lambda^j (\CC^N\otimes (\tau_{\lm'}/\tau_{\lm+2})^\vee)[j]\otimes
  (\det\tau_{\lm'})^{k-N+j-i+1}\otimes (\det\tau_{\lm''})^{j-i},
\end{multline}
where \(i=1,\dots, k\) and \(j=N-k+i,\dots,N\). In the discussion below we ignore \(\mathbf{q}\) shifts.

The fibers of \(\pi_{-}\) over \(\Delta^-=\pi_{-}(E)\) are projective spaces \(\PP^{k-1}\). On the other hand the restriction of the above sheaf on
a fiber of \(\pi_{-}\) over a point of \(\Delta^-\) is the vector bundle \((\det\tau_{\lm+2})^{i}\otimes \Lambda^j\CC^N[j]\). Since \((\det\tau_{\lm+2})=\calO(-1)\) as
line bundle on \(\PP^{k-1}\) the only terms that survive the push-forward \(\pi_{-*}\) are the terms with \(i=k\). 

Now we apply the push forward \(p_{-*}\) to the term that survived the push-forward \(\pi_{-*}\).  This is the term with \(i=k, j=N\):
\[ (\det \tau_{\lm+2})^{k}\otimes \calO_E\otimes\det(\tau_{\lm'})^{-N+1}\otimes (\det\tau_{\lm''})^{N-k}[n],\]
On the other hand
\((\det\tau_{\lm+2})^k\otimes(\det\tau_{\lm'})^{k+1}\)  is the relative dualizing sheaf for the fibers of \(\pi_-\) over \(\Delta^-\).  Hence the push-forward along the map \(\pi_{-*}\) of the
term for \(i=k,j=N\) is the line bundle
\((\det\tau_{\lm'})^{k-N}\otimes (\det\tau_{\lm''})^{N-k}\otimes \calO_{\Delta^-}[k-N-1]\).
Thus the push-forward along \(p_-\circ \pi_-\) of the term with \(i=k,j=N\) yields, with a correct \(\mathbf{q}\)-degree shift, \( \mathbf{q}^{2N}\calO_{\Delta}[k-N-1]\).


For the other map:
\begin{multline}
  (\calL^\vee)^{-i}\otimes (\det\tau_{\lm+2})^{k-N}\otimes
  (\det\tau_{\lm''})^{N-k}\otimes \det\tau_{\lm'}\otimes  \calO_E\otimes\Lambda^jV[-j]\simeq
  (\det\tau_{\lm-2})^{i}\\\otimes (\det\tau_{\lm+2})^{k-N+i}\otimes \calO_E \otimes \Lambda^j (\CC^N\otimes (\tau_{\lm'}/\tau_{\lm+2})^\vee)[j]\otimes
  (\det\tau_{\lm'})^{1-i}\otimes (\det\tau_{\lm''})^{N-k-i}
  \end{multline}
where \(i=1,\dots,N-k\), \(j=1,\dots,N-k-i+1\).

The fibers of the map \(\pi_+\) over \(\Delta^+=\pi_+(E)\) are projective spaces \(\PP^{N-k-1}\).
The restriction to a fiber of \(\pi_+\)   over a point \(\Delta^+\) is \((\det\tau_{\lm-2})^{i}\otimes \Lambda^j\CC^N[j]\). Thus the only term that survives is
\(i=N-k\), \(j=0\). This term is 
\[ (\det\tau_{\lm-2})^{N-k}\otimes \calO_E \otimes 
  (\det\tau_{\lm'})^{1-N+k}.
 \]

 The relative dualizing sheaf of \(\pi_+\) over \(\Delta^+\) is \((\det\tau_{\lm-2})^{N-k}\otimes 
  (\det\tau_{\lm'})^{1-N+k}\).
  Hence result of the push-forward \(\pi_{+*}\) of the term with \(i=N-k,j=0\) is \(\calO_{\Delta^+}[N-k-1]\).
  Thus the push-forward   \(p_{+*}\circ \pi_{-*}\) yields \(\mathbf{q}^0\calO_{\Delta}[N-k-1]\).

  \subsection{Lowest weight space}
\label{sec:lowest-weight-space}
Part of the arguments from the previous section hold for the case \(\lambda=-N\), \(k=N\). Indeed, the formulas \eqref{eq:kerEF}, \eqref{eq:kerEF-2}
for the functor \(\calF(\lm+2)\circ\calE(\lm)\) hold:
\[\calE(-N+2)\circ \calF(-N)=p_{+*}(\calO_{\hat{W}(N,N-2,N)}).\]

The image of \(p_{+}(\hat{W}(-N,-N+2,-N))\) is a subvariety of the diagonal \(\Delta\subset Y(-N)\times Y(-N)=\gl(N)\times \gl(N)\). Thus to
complete our proof we need to compute the push-forward \(p'_{+*}(\calO_{\hat{W}(-N,N+2,-N)})\) for the map \(p'_+:M'\to Y(N)\) defined by
\(p'_+(\calV_{N-1},X)=X\). The image \(p'_+(\hat{W}(-N,-N+2,-N))\) consists of the degenerate matrices thus we
expect:
\begin{equation}\label{eq:OO}
  p'_{+*}(\calO_{\hat{W}(-N,N+2,-N)})=[\mathbf{q}^{2N}\calO\xrightarrow{\det}\calO].\end{equation}

To prove the last formula we observe that \(\hat{W}(-N,N+2,-N)\) is a subvariety of \(Gr(N-1,N)\times \gl(N)\) that is a zero section of the
tautological section \(s\) of the vector bundle \(\pi^*((\CC^N)^\vee\otimes \CC^N/\tau_{-N+2})\), where \(\pi\) is the projection onto \(Gr(N-1,N)\).
Thus we have \[p'_{+*}(\calO_{\hat{W}(-N,N+2,-N)})=pr_*\bigl(K(\pi^*(\CC^N\otimes (\CC^N/\tau_{-N+2})^\vee),s)\bigr),\]
where \(pr\) is the projection onto \(\gl(N)\).  Since \((\CC^N/\tau_{-N+2})^\vee=\calO(-1)\), we conclude that the only terms in the homological degrees
\(0\) and \(N\) survive the push-forward \(pr_*\) and resulting complex is exactly \eqref{eq:OO}

\section{Towards the categorification of $\Uqglmn$}
\label{sec:next}

An advantage of the algebro-geometric approach to the categorification of $\Uqgloo$ is that, conjecturally,
its framework can be extended easily to $\Uqglmn$.

Consider a \emph{permuted} standard basis $e_1,\ldots,e_{m+n}$ of the defining $\glmn$ representation $V=\ICmn$. "Permuted" means that we do not assume that the first $m$ vectors are even, whereas the last $n$ ones are odd. Rather, we specify the parity map $\xf\colon \{1,\ldots,m+n\}\rightarrow \ZZ_2$ such that $e_i$ is even if $\xf(i) = 0$ and $e_i$ is odd if $\xf(i) = 1$, so $|\xf^{-1}(0)| = m$ and $|\xf^{-1}(1)| = n$. The weight subspaces $(\grVoN)_{\bflam}\subset \grVoN$ are indexed by weights $\bflam = (\lam_1,\ldots,\lam_{m+n})$ so that the weight of $e_i$ is  $\lam_j(e_i) = \delta_{ij}$, where $\delta_{ij}$ is the Kronecker symbol.

To a weight $\bflam$ we associate a partial flag variety
\[
\Fl(\bflam) = \{
\bfgrV = (0=\grV_0 \subset \grV_1 \subset\ldots\subset \grV_{m+n} = \ICN)\;|
\; \dim \grV_i = \lam_1 + \cdots + \lam_i
\}
\]
and (the total space of) a bundle $Y(\bflam)\rightarrow \Fl(\bflam)$ defined similar to $Y(\lam)$:
\begin{equation}
\label{eq:dfYbf}
Y(\bflam) = \Biggl\{ (\bfgrV,\grX)\in \Fl(\bflam)\times\End(\ICN)\;|\; \grX(\grV_i) \subset
\begin{cases}
\grV_{i},&\text{if $e_i$ is even,}
\\
\grV_{i-1},&\text{if $e_i$ is odd.}
\end{cases}
\Biggr\}
\end{equation}
The group $\Csq$ again acts on $Y(\bflam)$ by scaling the fibers: $X\mapsto t^2 X$,
and to a weight subspace of the weight $\bflam$ we associate the category $ \DbCs\bigl(Y(\bflam)\bigr) $.

Note that each basis vector $e_i$ corresponds to a pair of subspaces $\grV_{i-1}\subset\grV_i$ and a subspace $\grV_i$ corresponds to a pair of basis vectors $e_i,e_{i+1}$.

Similar to the categorification of $\gln$, the functors $\calE_{i,i+1}$ and $\calF_{i+1,i}$ categorifying the corresponding generators of $\glmn$ are expected to be presented by \FM\ kernels supported on the subvarieties $W(\bflam,\bflam')\subset Y(\bflam)\times Y(\bflam')$, where the weights $\bflam$ and $\bflam'$ are related by the action of $E_{i,i+1}\subset\glmn$:
\[
\lam'_j = \begin{cases}
\lam_j - 1,&\text{ if  $j=i$,}
\\
\lam_j + 1, &\text{ if $j=i+1$,}
\\
\lam_j, &\text{otherwise}.
\end{cases}
\]
Namely,
\[
W(\bflam,\bflam') = \{
(\bfgrV,\grX,\bfgrV',\grX')\;|\;\grV_i\subseteq \grV_i',\;\grX = \grX'
\}.
\]
The remaining challenge is to select the appropriate line bundles supported on $W(\bflam,\bflam')$ and to prove the $\Uqglmn$ relations between the \FM\ kernels.

One could further extend this construction to the tensor product of the \emph{symmetric} powers of $\grV$ by using the Slodowy slices. Namely, let $\Slbmu\subset \glN$ be the Slodowy slice to the orbit of the nilpotent matrix whose Jordan blocks have the sizes $\mu_i$ of an ordered partition $\bfmu=(\mu_1,\ldots,\mu_{\xl})$ such that $\mu_1\leq\cdots\leq \mu_{\xl}$ and $\mu_1+\cdots+\mu_{\xl} = N$. Then the category $\DbCs\bigl( Y(\bflam,\bfmu)\bigr)$ for
\[
Y(\bflam,\bfmu) = 
\Biggl\{ (\bfgrV,\grX)\in \Fl(\bflam)\times\Slbmu\;|\;\grX(\grV_i) \subset
\begin{cases}
\grV_{i},&\text{if $e_i$ is even,}
\\
\grV_{i-1},&\text{if $e_i$ is odd.}
\end{cases}
\Biggr\}
\]
should represent the weight $\bflam$ subspace of the tensor product of symmetric powers
\begin{equation}
\label{eq:tensbfmu}
\rmS^{\bfmu} V = \rmS^{\mu_1} V\otimes\cdots\otimes \rmS^{\mu_{\xl}} V.
\end{equation}
Here symmetrization is carried in accordance with the parity, that is, for example, $\rmS^\mu \IC^{0|n} \cong \Lambda^\mu \IC^n$. 

Conjecturally, one could extend the categorification farther to include symmetric powers of the \emph{shifted} defining representation $\IC^{n|m}$ and of their duals into the categorified tensor product~\eqref{eq:tensbfmu}. This would require the use of more general bow-arrow quiver varieties with some arrow and bow edges being "Legendre transformed" as described in~\cite{RR}, with matrix factorizations replacing coherent sheaves.

Note that from the perspective of the definition~\eqref{eq:dfYbf}, the  construction in \cite{CautisKamnitzerLicata10} described the categorification of the tensor power of the \emph{odd} version of the defining representation of $\glt$. That is, their $\glt$ is the $\glzt$ (rather than $\gltz$) member of the $\glmn$ family and their $V$ is $\IC^{0|2}$ rather than $\IC^{2|0}$. As a consequence, their construction extends to the categorification of the tensor product of the exterior powers of $\IC^2$ appearing as symmetric powers of $\IC^{0|2}$. The categorification of the symmetric powers of the defining $\IC^2$ representation of $\gl(2)$ would require presenting $\glt$ as $\gltz$ and hence the use of the "parabolic" rather than "nilpotent"  bundle over $\GrkN$:
\[
\YNlm=\{(\grV,\grX)\in\GrkN\times\End(\ICN)\,|\,\grV\subset \IC^N,\;\dim \grV = k,\;
\grX(\grV)\subset\grV\},
\]
instead of $\YNlm = \rmTs\GrkN$.


\end{document}